\documentclass[11pt]{article}

\usepackage{amsfonts}
\usepackage{amsmath}
\usepackage{pdfsync}
\usepackage{enumerate}

\usepackage[all]{xy}
\usepackage[english]{babel}
\usepackage[T1]{fontenc}
\usepackage{graphicx,color}
\usepackage{xcolor}
\colorlet{blu1}{blue!70!black}
\colorlet{blu2}{blue!50!black}
\colorlet{blu3}{blue!70!red}
\colorlet{blu4}{blue!60!green}
\colorlet{red1}{red!80}
\colorlet{red2}{red!50!black}
\colorlet{red3}{red!70!yellow}
\colorlet{red4}{red!50!yellow}
\colorlet{yel1}{yellow!50!black}
\colorlet{yel3}{yellow!20!blue}
\colorlet{gre1}{green!60!blue}
\colorlet{gre2}{green!60!black}
\colorlet{gre3}{green!40!black}

\newtheorem{theorem}{Theorem}
\newtheorem{definition}{Definition}
\newtheorem{lemma}{Lemma}
\newtheorem{remark}{Remark}

\newtheorem{proposition}{Proposition}
\newtheorem{example}{Example}

\newtheorem{proof}{Proof}

\numberwithin{equation}{section}

\title{Reducing Kolmogorov’s Reversibility Criterion Via a Basis for the Interaction Graph Kernel}
\author{ Cruz de la Rosa M.A. \& Guerrero-Poblete F. \\
\small Universidad Aut\'onoma Metropolitana, Iztapalapa. \\ 
\small{Av. San Rafael Atlixco, No. 186, Col. Leyes de Reforma $1^{\text{a}}$ Sección,} \\
\small {Iztapalapa, C.P. 09310, CDMX. M\'exico.}\\
\small \texttt{marko@xanum.uam.mx} \\
\small \texttt{poblete@xanum.uam.mx}
}

\date{}
\begin{document}

\maketitle

\begin{abstract}
\noindent In this work, we define the interaction graph for a continuous-time Markov chain to analyze its dynamic structure. Under certain assumptions, we demonstrate that verifying Kolmogorov's reversibility criterion reduces to checking it for a cycle basis of the incidence matrix kernel associated with the interaction graph. This result provides an efficient tool for verifying reversibility in Markov continuous chains.
\end{abstract}

\noindent \textbf{Keywords:} Continuous-time Markov chains, equilibrium distribution, Kolmogorov's reversibility criterion, interaction graph.\\
\noindent \textbf{Mathematics Subject Classification codes:} 60J27.

\section{Introduction}\label{sec:Introd}

Since their emergence in the early $20$th century, Markov chains have been used to model memoryless stochastic phenomena in different fields such as biology, physics and economics, among others. One of the most significant aspects both theoretically and practically is that of \textit{equilibrium}. An Equilibrium distribution is one that remains invariant over time.

Kolmogorov's reversibility criterion provides necessary and sufficient conditions for the existence of an equilibrium distribution. Roughly speaking, this criterion asserts that, for any cycle, the probability of traveling it in one direction is equal to the probability of traveling it in the opposite direction. In practice, verifying this condition for all cycles can be a difficult task. In this paper, we show that it suffices to check the criterion for a reduced set of cycles of minimal length. These cycles form a basis for the incidence matrix kernel of the interaction graph associated with the Markov chain, notion introduced in this work and inspired by the definition of interaction graphs for weak-coupling limit type quantum Markov semigroups studied in \cite{AccardiFQ} and \cite{Marco-Fer-Julio}.

This paper is organized as follows: after the Introduction, Section 2 presents the concept of equilibrium (detailed balance). In Section 3, we define the interaction graph associated with a Markov chain and provide a basis for the kernel of the incidence matrix in terms of cycles of minimal length. In Section 4,  we present the main result: to verify Kolmogorov's reversibility criterion, it suffices to do so for a basis of the incidence matrix kernel of the interaction graph. Finally, we illustrated these results with an example.

\section{Equilibrium and Detailed Balance}

We start by recalling the definition of equilibrium and detailed balance for a continuous time Markov chain.  

\begin{definition} Given a continuous time Markov chain $\{X_t\}_{t\geq 0}$ with state space $S$. We say that the chain is in equilibrium, if for any choice  $x_0,...,x_n \in S$ and $0\leq t_{0}<t_{1}<\cdots<t_{n}$

$$\mathbb{P}(X_{t_0}=x_0,...,X_{t_n}=x_n)=\mathbb{P}(X_{t_0}=x_n,...,X_{t_n}=x_0)$$
\end{definition}
(See \cite{Durretti}, \cite{Mu-Fa-Chen} and \cite{Stroock} for a detailed review).

\noindent Following the standard notation in Markov chain theory, we shall denoted by $q_{ij}$ the matrix elements of the infinitesimal generator $Q$ and by  $\pi_{i}$ to $i$-nt element of the invariant distribution $\pi$. Recall that a distribution $\pi$ is invariant if only if $\pi Q=0$ and it is fulfilled 
\begin{enumerate}
\item $\displaystyle q_{ij}\geq 0$,\ for all  $i\neq j$ 
\item $\displaystyle q_{ii}=-\sum_{j: j\neq i}q_{ij}$,\ for all   $i\in S$
\end{enumerate}
\noindent For more details, see \cite{Norris}.
  
\begin{definition} Given a continuous time Markov chain $\{X_t\}_{t\geq 0}$ with state space $S$, infinitesimal generator $Q$
and invariant distribution $\pi$, we say that the chain satisfies the detailed balance condition if for all $i\neq j$,
\begin{equation}
\label{balance-detallado}\pi_{i} q_{ij}=\pi_{j}q_{ji}
\end{equation}
\end{definition}

\noindent In an equivalent way, detailed balance means that $\pi_{i} q_{ij}-\pi_{j}q_{ji}=0$, for all $i\neq j$; these quantities are known as \textit{currents}. A well known fact is that the equilibrium condition is equivalent to the detailed balance condition, (see \cite{Stroock}), nevertheless the last one is easier to handle.\\

\noindent The following lemma will be useful for the next section.
\begin{lemma}
Let $\{X_t\}_{t\geq 0}$ a continuous time Markov chain  with state space $S$ and infinitesimal generator $Q$, a distribution $\pi$ is invariant if only if for all $j\in S$
\begin{eqnarray}\label{Suma-DB}
\sum_{i<j}(\pi_{i}q_{ij}-\pi_{j}q_{ji})+\sum_{i>j}(\pi_{i}q_{ij}-\pi_{j}q_{ji})=0
\end{eqnarray}
\end{lemma}
\begin{proof}
Since $\pi$ is invariant if only if $\pi Q=0$, in the $j$-th entry we have
\begin{eqnarray}\nonumber 
(\pi Q)_{j}&=&\sum_{i}\pi_{i}q_{ij}=\sum_{i: i\neq j}\pi_{i}q_{ij}+\pi_{j}q_{jj}=\sum_{i: i\neq j}\pi_{i}q_{ij}-\pi_{j}\sum_{i: i\neq j} q_{ji}\\
\nonumber &=&\sum_{i: i\neq j}(\pi_{i}q_{ij}-\pi_{j}q_{ji})=\sum_{i<j}(\pi_{i}q_{ij}-\pi_{j}q_{ji})+\sum_{i>j}(\pi_{i}q_{ij}-\pi_{j}q_{ji})
\end{eqnarray} 
\end{proof}

\section{Interaction graph of a Markov chain}

\noindent In this section we shall define the interaction graph of a Markov chain. From now on, we shall suppose that the state space $S$ is finite and for all $i\neq j$  we have $q_{ij}>0$, in this way, equation ($\ref{Suma-DB}$) takes the form:
\begin{eqnarray}\label{Currents-finite}
\sum_{i=1}^{j-1}(\pi_{i}q_{ij}-\pi_{j}q_{ji})+\sum_{i=j+1}^{N}(\pi_{i}q_{ij}-\pi_{j}q_{ji})=0.
\end{eqnarray}
\noindent 
For $i\neq j$, we define $J_{i,j}:=(\pi_{i}q_{ij}-\pi_{j}q_{ji})$. It is clear that $-J_{j,i}=J_{i,j}$. Hence, condition $(\ref{Currents-finite})$ can be written as 
\begin{eqnarray} \label{Currents-Jij}
\sum_{i=1}^{j-1}J_{ij}-\sum_{i=j+1}^{N}J_{ji}=0,\qquad \forall j\in S.
\end{eqnarray} 

\noindent   The following definition is inspired by \cite{AccardiFQ} and \cite{Marco-Fer-Julio}, where the interaction graph is defined for quantum Markov semigroups of weak coupling limit type.  

\begin{definition}
Let $\{X_t\}_{t\geq 0}$ a Markov chain with state space $S=\{1,2,\dots, N\}$ and infinitesimal generator  $Q$. The interaction graph associated with the Markov chain is the graph $G(V,E)$, where the set of vertices is $V=S$  and the set of edges is $E=\{(i,j)\in S\times S: i<j,\ q_{ij}>0\}$. 
\end{definition}

\noindent In this work we establish that for any $(i,j)\in E$, the edge comes out from $i$ and comes in to $j$,  so, the interaction graph is in fact a directed graph (digraph). Given a digraph $G(V,E)$, its incidence matrix has dimension ${|V|\times |E|}$, that is, its rows are indexed by vertices and its columns by edges. The elements of the incidence matrix are such that in the $j$-th column and $i$-th row have $-1$, if the $j$-th edge comes out from vertex $i$, $1$ if the $j$-th edge comes in to vertex $i$ and zero in other case.

\noindent 
From now on, we shall denote by $\Gamma$ to the incidence matrix of the interaction graph $G(V,E)$. Notice that if the graph is complete, the set of edges has cardinality  $\binom{N}{2}$.  
For the pair  $(i,j)$, $(i',j')$ the lexicographic order establishes that $(i,j)<(i',j')$ if $i<j$, or, when $i=j$ one has that $j<j'$.
The set of edges $E$ shall be equipped with the lexicographic order, for such order the function $\theta_N:E\to \{1,\dots,\binom{N}{2}\}$ given by 
\begin{eqnarray}
\theta_{N}(i, j)=S_{N}(i)+j-i, \quad \text{where}\  S_{N}(i):=\sum_{t=1}^{i-1}(N-t)=(i-1)\big(N-\frac{i}{2}\big)\nonumber\\
\end{eqnarray}
is increasing and bijective  (see \cite{Marco-Fer-Julio} for more details), we shall also use indistinctly $\theta$ or $\theta_{N}$. We define the \textit{currents vector } as the vector $J\in \mathbb{R}^{|E|}$, whose entries are precisely the currents $J_{i,j}$ with the lexicographic order, i.e., 
$J_{i,j} \rightarrow J_{\theta(i,j)}$, namely
\begin{eqnarray}
J:=\left(\begin{array}{c}
J_{1,2}\\
\vdots\\
J_{1,N}\\
J_{2,3}\\
\vdots\\
J_{N-1, N}
\end{array}\right)\rightarrow \left(\begin{array}{c}
J_{1}\\
\vdots\\
J_{N-1}\\
J_{N}\\
\vdots\\
J_{N\choose 2}
\end{array}\right)
\end{eqnarray} 

\noindent The following theorem shows that condition (\ref{Currents-Jij}) is equivalent to the currents vector $J$ being an element in $ker(\Gamma)$. 

\begin{theorem}
Given  $\{X_t\}_{t\geq 0}$ a Markov chain with state space $S=\{1,2,\dots, 
N\}$ and infinitesimal generator $Q$, the following is equivalent 
\begin{enumerate}
\item $\Gamma J=0$.
\item $\displaystyle \sum_{i=1}^{j-1}J_{ij}-\sum_{i=j+1}^{N}J_{ji}=0$, for all $j\in S$.
\end{enumerate}
\end{theorem}

\begin{proof}
For $(i', j')\in E$ such that  $\theta(i',j')=k$, one has that $\Gamma$ has entries   
\begin{equation}
\Gamma_{j,k}=\left\{
              \begin{array}{rccl}
                -1 & &\text{if}\ i'=j,& \text{``since the edge comes out from $i'$''} \\
                1 &&\text{if}\ j'=j, & \text{``since the edge comes in to $j'$''} \\
                0 && \text{other case.}&
              \end{array}
            \right.
\end{equation}

\noindent On the other hand, $\Gamma J=0$ if only if  $(\Gamma J)_{j}=0$ for all $j$, that is,
\begin{eqnarray}\nonumber 
0&=& \sum_{k} \Gamma_{j,k} J_{k}=\sum_{(i',j'):i'<j'}\Gamma_{j,\theta(i',j')}J_{\theta(i',j')} \\ 
&=& \sum_{j':j'>j}\Gamma_{j,\theta(j,j')}J_{\theta(j,j')}+ \sum_{i':i'<j}\Gamma_{j,\theta(i',j)}J_{\theta(i',j)} \\ \nonumber
\end{eqnarray} 
by renaming variables, the above is satisfied if only if 

\begin{eqnarray}
0=\sum_{i:i>j}\Gamma_{j,\theta(j,i)}J_{\theta(j,i)}+\sum_{i:i<j}\Gamma_{j,\theta(i,j)}J_{\theta(i,j)}=\sum_{i:i>j}J_{\theta(j,i)}-\sum_{i:i<j}J_{\theta(i,j)}
\end{eqnarray}

\noindent if and only if
\begin{eqnarray}
\sum_{i:i<j}J_{ij}-\sum_{i:i>j}J_{ji}=0.
\end{eqnarray}
\end{proof}

\noindent The following theorem gives us an easy way to compute the dimension of $ker(\Gamma)$. For a proof see \cite{Busacker}.

\begin{theorem}\label{dim}
The rank of the incidence matrix associated with a graph with $N$ vertices is $N-p$, where $p$ is the number of connected components.
\end{theorem}

\begin{remark}
If the graph is complete, by Theorem $\ref{dim}$, the rank of the matrix $\Gamma$ is $N-1$, so the kernel has dimension $\binom{N}{2}-(N-1)=\binom{N-1}{2}$.
\end{remark}

\noindent For an arbitrary graph, the minimal length of the cycles is three. We shall denote as $C_{(i, i+1,i+1+j)}$ to the minimal length cycle that connects the vertices $i$, $i+1$ and $i+1+j$. Given the set of the triplets $E':=\{(i,i+1,i+1+j)\,:\, 1\le i\le N-2,\, 1\le j\le N-(i+1)\}$; if $E'$ is  equipped with the lexicographic order, then the function ${\theta'_N:E'\to \{1,\dots,\binom{N-1}{2}\}}$, given by
\begin{eqnarray}\label{Orden-triple}
\theta'_N(i,i+1,i+1+j)=\theta_N(i+1,i+1+j)-(N-1)
\end{eqnarray}
is increasing and bijective (see \cite{Marco-Fer-Julio} for more details).

\noindent The following theorem states that the set of minimal length cycles conforms a basis for $ker(\Gamma)$ and gives an explicit formula for the vector coordinates of the cycles.

\begin{theorem}
For $1\leq i\leq N-2$ and $1\leq j\leq N-(i+1)$, let the mi\-ni\-mal length cycles $C_{(i,i+1,i+1+j)}$ given by the following formula: for $k\in \{1,2,\dots, \binom{N}{2}\}$,
\begin{equation}
C_{(i, i+1,i+1+j)}(k)=\left\{
              \begin{array}{rcccl}
                1 & &\text{if} && k=(i-1)(N-\frac{i}{2})+1, \\
                  & &\text{or} && k=i(N-\frac{i+1}{2})+j, \\
                -1 &&\text{if} &&k=(i-1)(N-\frac{i}{2})+j+1, \\
                0 &&  &&\text{other case.}
              \end{array}
            \right.\label{tray cerradas}
\end{equation}
then the set $\mathcal{C}:=\left\{C_{(i,i+1,i+1+j)}:\, 1\leq i\leq N-2,\, 1\leq j\leq N-(i+1) \right \}$ is a basis for $Ker(\Gamma)$.
\end{theorem}
 
\begin{proof}
\noindent See \cite{Marco-Fer-Julio}.
\end{proof}

\noindent For each triplet in $E'$, from equation (\ref{Orden-triple}) we identify the cycles 
$$C_{(i, i+1,i+1+j)} \rightarrow C_{\ell}\quad \textrm{if}\ \theta'_N(i,i+1,i+1+j)=\ell$$
\noindent  In such a way, if $J\in ker(\Gamma)$ 
\begin{eqnarray}\label{J_lin_com}
J=d_1 C_{1}+d_2 C_{2}+\dots +d_{N-1\choose 2} C_{N-1\choose 2}
\end{eqnarray}
where $d_{\ell}$ is	 the value of the currents in the cycle $C_{\ell}$. If the chain is in equilibrium, then $d_{\ell}=0$ for all $1\leq \ell \leq {N-1\choose 2}$.

\section{Kolmogorov's reversibility criterion}

\noindent Let us now recall the Kolmogorov's reversibility criterion for Markov chains, which provides necessary and sufficient conditions for a chain to be in equilibrium. 

\begin{proposition}\label{Kolmogorov-Criterio}
A necessary and sufficient condition for a continuous time Markov chain with infinitesimal generator $Q$ to be in equilibrium is, 
\begin{enumerate}[a)]
\item If $q_{ij}>0$, then $q_{ji}>0$.
\item For all cycle $x_{0},x_{1},\dots,x_{n}=x_{0}$ such that $\prod_{i=1}^{n} q_{i,i-1}>0$, then 
\begin{equation}\label{Kolmo-cond}
\prod_{i=1}^{n} \frac{q_{i-1,i}}{q_{i,i-1}}=1.
\end{equation}
\end{enumerate}
\end{proposition}

\noindent See \cite{Durrett} for a proof.

\noindent Now, we present the main result of the paper.

\begin{theorem}
Given a continuous time Markov chain with finite state space $S$, such that $|S|=N$ and infinitesimal generator $Q$ such that
 $q_{ij}>0$ for all $i\neq j$; a necessary and sufficient condition for a chain to be in equilibrium is that for all triple
$(i,i+1,i+1+j)$ with $1\le i\le N-2$ and $1\le j\le N-(i+1)$, the determinants  
\begin{eqnarray}
\Delta_{(i,i+1,i+1+j)}:=&\left|\begin{array}{ccc}
q_{i,i+1}&-q_{i+1, i}&0\\
q_{i, i+1+j}&0&-q_{i+1+j, i}\\
0&q_{i+1,i+1+j}&-q_{i+1+j, i+1}
\end{array}\right|
\end{eqnarray}
are equal to zero.
\end{theorem}
\begin{proof}
\noindent We shall prove that the assumptions of this theorem imply Kolmogorov's reversibility criterion. First of all, $q_{ij}>0$ for all $i\neq j$  implies $a)$ in  Proposition \ref{Kolmogorov-Criterio}.

\noindent On the other hand, condition  b) in Proposition \ref{Kolmogorov-Criterio} for a minimal length cycle $C_{(i, i+i, i+1+j)}$ is

\[\frac{q_{i,i+1}}{q_{i+1,i}}\frac{q_{i+1,i+1+j}}{q_{i+1+j,i+1}}\frac{q_{i+1+j,i}}{q_{i,i+1+j}}=1,\]

\noindent in an equivalent way  
\begin{eqnarray}
\nonumber 0&=&q_{i,i+1} q_{i+1,i+1+j} q_{i+1+j,i}- q_{i,i+1+j} q_{i+1+j,i+1}  q_{i+1,i}\\ 
&=&\left|\begin{array}{ccc}
q_{i,i+1}&-q_{i+1, i}&0\\
q_{i, i+1+j}&0&-q_{i+1+j, i}\\
0&q_{i+1,i+1+j}&-q_{i+1+j, i+1}
\end{array}\right|=\Delta_{(i,i+1,i+1+j)}
\end{eqnarray} 

\noindent Since $\mathcal{C}:=\left\{C_{(i,i+1,i+1+j)}:\, 1\leq i\leq N-2,\, 1\leq j\leq N-(i+1) \right \}$ is a basis for the set of cycles, the result it follows.
\end{proof}

\begin{example}
Consider a Markov chain with four states and generator  
\[
Q=\left(\begin{array}{cccc}
q_{11}&q_{12}&q_{13}&q_{14}\\
q_{21}&q_{22}&q_{23}&q_{24}\\
q_{31}&q_{32}&q_{33}&q_{34}\\
q_{41}&q_{42}&q_{43}&q_{44}
\end{array}\right),\qquad q_{ii}=-\sum_{j:j\neq i}q_{ij}.
\]

\noindent Its incidence matrix and interaction graph are
\[
\Gamma=\left( \begin{array}{rrrrrr}
-1&-1&-1&0 & 0 &0\\
1 & 0& 0&-1& -1&0\\
0 & 1& 0& 1& 0&-1\\
0 & 0& 1&0& 1&1\\
\end{array}\right),
\]

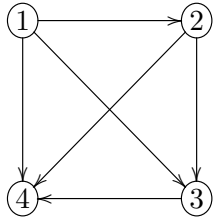
\begin{figure}[htb]
\[\xymatrix{ *+[o][F]{1} \ar[rr] \ar[rdrd] \ar[dd] & &*+[o][F]{2} \ar[ldld] \ar[dd] \\ & & \\
*+[o][F]{4} & & *+[o][F]{3} \ar[ll] }  \]
\caption{Interaction graph for a four-state chain.}
\end{figure}

\noindent From $(\ref{tray cerradas})$ it follows that a basis for the kernel of $\Gamma$ is $ \mathcal{C}=\{C_{1}, C_{2}, C_{3}\}=\{C_{(1,2,3)}, C_{(1,2,4)}, C_{(2,3,4)} \}$, where
\[
C_{(1,2,3)}=\left(\begin{array}{r}
1\\-1\\0\\1\\0\\0
\end{array}  \right),\qquad 
C_{(1,2,4)}=\left(\begin{array}{r}
1\\0\\-1\\0\\1\\0
\end{array}  \right),\qquad 
C_{(2,3,4)}=\left(\begin{array}{r}
0\\0\\0\\1\\-1\\1
\end{array}  \right)
\]
The Kolmogorov's reversibility criterion for the cycles in the basis $\mathcal{C}$ is 
\[
\frac{q_{12} q_{23} q_{31}}{q_{21} q_{32} q_{13}}=1,\qquad \frac{q_{12} q_{24} q_{41}}{q_{21} q_{42} q_{14}}=1,\qquad \frac{q_{23} q_{34} q_{42}}{q_{32} q_{43} q_{24}}=1. 
\]
Notice that if a minimal cycle is traveled in opposite way, let say $C_{(i,i+1+j, i+1)}$, we have that 
\[
1=\frac{q_{i, i+1+j} q_{i+1+j,i+1} q_{i+1, i}}{q_{i+1+j,i} q_{i+1,i+1+j} q_{i,i+1}}=\left(\frac{ q_{i,i+1} q_{i+1,i+1+j} q_{i+1+j,i}}{q_{i+1, i} q_{i+1+j,i+1} q_{i, i+1+j}}\right)^{-1},
\] 
which is the inverse of the reversibility criterion for the cycle $C_{(i, i+1, i+1+j)}$.
Thus, for any other cycle, say  
\[
C_{(1,3,4)}=\left(\begin{array}{r}
0\\1\\-1\\0\\0\\1
\end{array}  \right)
\]
It is verified that  $C_{(1,3,4)}=-C_{(1,2,3)}+C_{(1,2,4)}+C_{(2,3,4)}$, the coordinates $(-1, 1,1)$ of $C_{(1,3,4)}$ in the ordered basis  $\mathcal{C}$, tell us that this is equivalent to traveling the cycle $C_{(1,2,3)}$ 
in reverse way followed by the cycles $C_{(1,2,4)}$ and $C_{(2,3,4)}$, so
\[
1=\Big(\frac{q_{12} q_{23} q_{31}}{q_{21} q_{32} q_{13}}\Big)^{-1}\frac{q_{12} q_{24} q_{41}}{q_{21} q_{42} q_{14}}\frac{q_{23} q_{34} q_{42}}{q_{32} q_{43} q_{24}}=
\frac{  q_{13} q_{34} q_{41}  }{q_{31} q_{43} q_{14} }
\]
i.e., Kolmogorov's reversibility criterion is satisfied for the cycle $C_{(1,3,4)}$. The same holds for the cycle $C_{(1,2,3,4)}=(1,0,-1,1,0,1)^{t}=C_{(1,2,4)}+C_{(2,3,4)}$ which has coordinates $(0,1,1)$ 
\[
1=\frac{q_{12} q_{24} q_{41} q_{23} q_{34} q_{42}}{q_{21} q_{42} q_{14} q_{32} q_{43} q_{24}}=\frac{q_{12} q_{23} q_{34}  q_{41}  }{q_{21}  q_{32} q_{43}  q_{14}}.
\]
\end{example}

\end{document}